\documentclass[10pt]{amsart}

\usepackage{amsmath,amsfonts,amssymb,amsthm,amsaddr}
\theoremstyle{definition}

\newtheorem{theorem}{Theorem}
\newtheorem{proposition}[theorem]{Proposition}
\newtheorem{lemma}[theorem]{Lemma}
\newtheorem{definition}[theorem]{Definition}
\newtheorem{remark}[theorem]{Remark}

\numberwithin{theorem}{section}

\usepackage{tikz}
\usetikzlibrary{arrows,calc}

\newcommand{\ran}{\mathsf{ran}}

\newcommand{\D}{\mathsf{D}^2_2}
\newcommand{\SRT}{\mathsf{SRT}^2_2}
\newcommand{\SPT}{\mathsf{SPT}^2_2} 
\newcommand{\SIPT}{\mathsf{SIPT}^2_2}
\newcommand{\lW}{\le_{\rm W}}
\newcommand{\lsW}{\le_{\rm sW}}

\newcommand{\lsc}{\le_{\rm sc}}

\begin{document}
\title[Strong Reductions \& Stable Ramsey's Theorem]{Strong Reductions between Relatives of the Stable Ramsey's Theorem}
\author{David Nichols}
\address{University of Connecticut \\ {\tt david@davidmathlogic.com}}
\thanks{The author is indebted to Damir Dzhafarov, Ludovic Patey, and Reed Solomon for input during the formation of this paper. This paper is part of the author's PhD thesis in progress at the University of Connecticut. The writing of this paper was supported in part by the summer fellowship granted by that institution's mathematics department.}
\maketitle

\begin{abstract} 
A complete analysis is given of the computable reductions that hold between $\mathsf{SRT}^2_2$, $\mathsf{SPT}^2_2$, and $\mathsf{SIPT}^2_2$. In particular, while $\mathsf{D}^2_2\le_{\rm sW}\mathsf{SIPT}^2_2\le_{\rm sW}\mathsf{SPT}^2_2\le_{\rm sW}\mathsf{SRT}^2_2$, it is shown that $\mathsf{SRT}^2_2\not\le_{\rm sc}\mathsf{SPT}^2_2\not\le_{\rm sc}\mathsf{SIPT}^2_2\not\le_{\rm sc}\mathsf{D}^2_2$. 
\end{abstract}

\section{Introduction}

Much recent work in reverse mathematics has centered on combinatorial principles related to Ramsey's theorem. This paper will examine four such principles and will make a detailed analysis of their logical relationships. All of these principles involve colorings, for which we provide the following definitions. 

\begin{definition} $~$
\begin{itemize}
	\item If $S\subseteq\omega$, then $[S]^2$ denotes the set of all 2-element subsets of $S$. 
	\item A \textit{2-coloring of pairs} is a function $f:[\omega]^2\to 2$. We hereafter write $f(x,y)$ for $f(\{x,y\})$ with $x<y$, and ``coloring'' should be understood to mean ``2-coloring of pairs.'' A 2-coloring of pairs is \textit{stable} if $\lim_u f(x,u)$ exists for all $x\in\omega$. 
\end{itemize}
\end{definition}

Given a coloring, we are interested in the existence of sets homogeneous for the coloring in the following senses. 

\begin{definition} $~$
\begin{itemize} 
	\item An infinite set $H\subseteq\omega$ is \textit{homogeneous} for a coloring $f$ if $f\upharpoonright[H]^2$ is constant. 
	\item An infinite set $H = H_L \oplus H_R$ is p-\textit{homogeneous} for a coloring $f$ if $f\upharpoonright H_L\times H_R$ is constant, and is \textit{increasing {\rm p}-homogeneous} if $f\upharpoonright\{\{x,y\}\mid x<y,\,x\in H_L,\,y\in H_R\}$ is constant. 
	\item An infinite set $H$ is \textit{limit homogeneous} for a coloring $f$ if $\lim_u f(\cdot,u)$ is total and constant on $H$. 
\end{itemize} 
\end{definition}

With these definitions in hand, we state and name the following principles, in whose relative strength we are interested. 

\newpage
\begin{definition} $~$
\begin{itemize} 
	\item $\mathsf{SRT}^2_2$ is the statement that for every stable 2-coloring of pairs $f:[\omega]^2\to 2$ there exists a set $H$ homogeneous for $f$. 
	\item $\mathsf{SPT}^2_2$ is the statement that for every stable 2-coloring of pairs $f:[\omega]^2\to 2$ there exists a set $H$ $p$-homogeneous for $f$. 
	\item $\mathsf{SIPT}^2_2$ is the statement that for every stable 2-coloring of pairs $f:[\omega]^2\to 2$ there exists a set $H$ increasing $p$-homogeneous for $f$. 
	\item $\mathsf{D}^2_2$ is the statement that for every stable 2-coloring of pairs $f:[\omega]^2\to 2$ there exists a set $H$ limit homogeneous for $f$
\end{itemize} 
\end{definition}

$\mathsf{SPT}^2_2$ and $\mathsf{SIPT}^2_2$ were first studied in the context of combinatorics by Erd\H{o}s and Rado in~\cite{erdos} and in the context of reverse mathematics by Dzhafarov and Hirst in~\cite{polarized}. 
Our interest will be in studying these principles in terms of computable and Weihrauch reductions, to whose definition we now turn. 
The principles $\mathsf{SRT}^2_2$, $\mathsf{SPT}^2_2$, $\mathsf{SIPT}^2_2$, and $\mathsf{D}^2_2$, together with many of the principles encountered in reverse mathematics, take the syntactic form 
$$\forall X (\Phi(X) \rightarrow \exists Y \Psi(X,Y),$$ 
where $\Phi$ and $\Psi$ are arithmetical predicates. We call principles of this form \textit{problems}, and given a problem $\mathsf{P}$ we call objects $X$ such that $\Phi(X)$ holds \textit{instances} of the problem and objects $Y$ such that $\Phi(X,Y)$ holds \textit{solutions} to the instance $X$ of the problem $\mathsf{P}$. Given this language, we define the following notions of reducibility between two problems. 

\begin{definition} 
Let $\mathsf{P}$ and $\mathsf{Q}$ be problems. 
\begin{itemize} 
	\item $\mathsf{P}$ is \textit{computably reducible} to $\mathsf{Q}$, written $\mathsf{P} \le_{\rm c} \mathsf{Q}$, if every instance $X$ of $\mathsf{P}$ computes an instance $\widehat{X}$ of $\mathsf{Q}$ such that whenever $\widehat{Y}$ solves $\widehat{X}$, $X\oplus\widehat{Y}$ computes a solution $Y$ to $X$. 
	\item $\mathsf{P}$ is \textit{strongly computably reducible} to $\mathsf{Q}$, written $\mathsf{P} \le_{\rm sc} \mathsf{Q}$, if every instance $X$ of $\mathsf{P}$ computes an instance $\widehat{X}$ of $\mathsf{Q}$ such that whenever $\widehat{Y}$ solves $\widehat{X}$, $\widehat{Y}$ computes a solution $Y$ to $X$. 
	\item $\mathsf{P}$ is \textit{Weihrauch reducible} to $\mathsf{Q}$, written $\mathsf{P} \le_{W} \mathsf{Q}$, if there are Turing functionals $\Phi$ and $\Gamma$ such that whenever $X$ is an instance of $\mathsf{P}$, $\Phi^X$ is an instance of $\mathsf{Q}$, and whenever $\widehat{Y}$ solves $\Phi^X$, $\Gamma^{X\oplus\widehat{Y}}$ solves $X$. 
	\item $\mathsf{P}$ is \textit{strongly Weihrauch reducible} to $\mathsf{Q}$, written $\mathsf{P} \le_{sW} \mathsf{Q}$, if there are Turing functionals $\Phi$ and $\Gamma$ such that whenever $X$ is an instance of $\mathsf{P}$, $\Phi^X$ is an instance of $\mathsf{Q}$, and whenever $\widehat{Y}$ solves $\Phi^X$, $\Gamma^{\widehat{Y}}$ solves $X$. 
\end{itemize} 
\end{definition} 

These four notions are related by the following implications and no others~\cite{hirschfeldt}. 

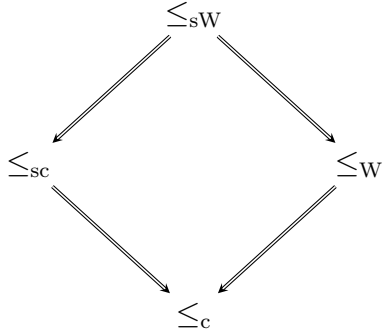
\begin{center}
\begin{tikzpicture}
\pgfmathsetmacro{\len}{2}

\node (sc) at (0,0) {\Large $\le_{\rm sc}$};
\node (sW) at ({1.1*\len},\len) {\Large $\le_{\rm sW}$};
\node (W) at ({1.1*2*\len},0) {\Large $\le_{\rm W}$}; 
\node (c) at ({1.1*\len},{-1*\len}) {\Large $\le_{\rm c}$};

\draw[double,-stealth] (sW) -- (sc); 
\draw[double,-stealth] (sW) -- (W); 
\draw[double,-stealth] (sc) -- (c);
\draw[double,-stealth] (W) -- (c);
\end{tikzpicture}
\end{center}

Each of these notions is meant to capture the intuitive idea that if one has the ability to solve $\mathsf{Q}$, one may in an algorithmic way use this ability to solve $\mathsf{P}$. While such ideas have been used at least implicitly in the reverse mathematics literature for some time, they were first presented as objects of study quite recently. Weihrauch and strong Weihrauch reducibility were introduced by Weihrauch in~\cite{weihrauch} in the context of degrees of discontinuity, and in the context of reverse mathematics were independently discovered by  Dorais, Dzhafarov, Hirst, Mileti, and Shafer~\cite{doraisetal}. Computable reducibility was developed by Dzhafarov~\cite{damirfirst}. 

Each of the principles $\mathsf{SRT}^2_2$, $\mathsf{SPT}^2_2$, $\mathsf{SIPT}^2_2$, and $\mathsf{D}^2_2$ is equivalent to each of the others over $\mathsf{RCA}_0$~\cite{chonglemppyang}. This paper will analyze the relationships between these principles in terms of the stronger notions of reducibility introduced in Definition 1.4. It is easy to see that $\D\lsW\SIPT\lsW\SPT\lsW\SRT$. Furthermore, we have 
\begin{proposition} 
$\mathsf{SRT}^2_2 \le_{\rm W} \mathsf{SPT}^2_2\le_W \mathsf{SIPT}^2_2$. 
\end{proposition}
\noindent We will use the following lemma to prove Proposition 1.5. 

\begin{lemma}
Fix $i<2$. There is a Turing functional $\Phi$ such that if $f:[\omega]^2\to 2$ is a stable 2-coloring of pairs and $L$ is an infinite set limit homogeneous for $f$ with color $i$, then $\Phi^{f\oplus L\oplus\{i\}}$ describes an infinite set $H$ homogeneous for $f$. 
\end{lemma}

\begin{proof}[Proof of Lemma 1.6] Fix $f$ and $L$. We may compute from $f\oplus L\oplus\{i\}$ an infinite set $H = \{h_0,h_1,h_2,\dots\}$ homogeneous for $f$ as follows. Let $h_0$ be the least element of $L$, $h_1$ the least element of $L$ such that $\{h_0,h_1\}$ is finite homogeneous for $f$ with color $i$, and $h_{k+1}$ the least element of $L$ such that $\{h_0,h_1,\dots,h_k\}$ is finite homogeneous for $f$ with color $i$. The stability of $f$ and the fact that $L$ is limit homogeneous for $f$ with color $i$ imply that this enumeration is computable from $f$, $L$, and $i$; this together with the ordering of enumeration implies that $H$ is computable in the given data.
\end{proof}

\begin{proof}[Proof of Proposition 1.5]
Fix a stable 2-coloring of pairs $f:[\omega]^2\to 2$ and an infinite set $I = I_0\oplus I_1$ increasing $p$-homogeneous for $f$. Let $i_0$ be the least element of $I_0$ and $i_1$ the least element of $I_1$ such that $i_0 < i_1$. Because $f$ is stable and $I$ is increasing $p$-homogeneous for $f$, $I_0$ is limit homogeneous for $f$ with color $f(i_0,i_1)$. Since $f(i_0,i_1)$ and $I_0$ are computable from $f\oplus I$, by the lemma $f\oplus I$ computes an infinite set $H$ homogeneous for $f$, hence the infinite set $H\oplus H$ $p$-homogeneous for $f$.
\end{proof}

By contrast to Proposition 1.5, Dzhafarov~\cite{combinatorialprinciples} showed that $\SRT\not\lW\D$ and $\SRT\not\lsc\D$. 
We strengthen the latter result by replacing $\SRT$ with $\SIPT$ as follows. 



\begin{theorem}
$\mathsf{SIPT}^2_2\not\le_{\rm sc}\mathsf{D}^2_2$.
\end{theorem}

This follows by a straightforward adaptation of the proof of Corollary 3.6 in~\cite{combinatorialprinciples}, but in addition we provide the following much simpler proof. 

\begin{proof}[Proof of Theorem 1.7]
We will choose $f:[\omega]^2\to 2$ be a non-computable instance of $\mathsf{SIPT}^2_2$ all of whose solutions compute $\emptyset'$. To that end, fix a c.e. approximation of $\emptyset'$, $\{X_s\}_{s\in\omega}$, with least modulus $\mu$. Now define 
$$f(x,y) = \begin{cases} 
0, & \quad\text{if }y-x\le {\rm max}\{\mu(z):\,z\le x\}; \\ 
1, & \quad\text{otherwise.}
\end{cases}$$
Then $\lim_u f(x,u) = 1$ for all $x\in\omega$. 
Let $H=H_L\oplus H_R$ be any solution to $f$ (so $H$ is increasing $p$-homogeneous for $f$ with color 1). Then if $z\in\omega$, we can compute from $H$ whether $z\in\emptyset'$ as follows: find the least $x\in H_L$ with $z\le x$, and the least $y > x$ with $y\in H_R$. Since $f(x,y) = 1$, $y-x > \mu(z)$. Thus $z\in\emptyset'$ if and only if $z\in X_{y-x}$. 

So every solution to $f$ computes $\emptyset'$. But every instance of $\mathsf{D}^2_2$---in particular, every instance computable from $f$---has a solution which does not compute $\emptyset'$ (\cite{coneavoidance}), hence does not compute any infinite set increasing $p$-homogeneous for $f$.
\end{proof}

Our focus in this paper is on the following two theorems, which together with Theorem 1.7 show that none of the implications in Proposition 1.5 can be strengthened to strong computable reductions.

\begin{theorem}[First Main Theorem]
$\mathsf{SRT}^2_2\not\le_{\rm sc}\mathsf{SPT}^2_2$.
\end{theorem}

\begin{theorem}[Second Main Theorem]
$\mathsf{SPT}^2_2\not\le_{\rm sc}\mathsf{SIPT}^2_2$.
\end{theorem}

The proofs of the two main theorems will appear later on in their own sections. In section 2 we will develop the combinatorial tool that will be central to those proofs. In section 3 we will prove the first main theorem, and in section 4 we will prove the second main theorem. Section 5 summarizes the results of this paper and, by combining them with the two results of~\cite{combinatorialprinciples} cited above, presents a complete analysis of $\D$, $\SIPT$, $\SPT$, and $\SRT$ in terms of computable and Weihrauch reductions.

\section{Tree Labeling}

To prove the main theorems we will use a pair of \textit{tree labeling} arguments. The tree labeling method was first introduced by Dzhafarov in~\cite{combinatorialprinciples} to prove that $\mathsf{COH}\not\le_{\rm sc}\mathsf{SRT}^2_2$, and has since been used by Dzhafarov, Patey, Solomon, and Westrick in~\cite{etal} to prove that $\mathsf{COH}\not\le_{\rm sc}\mathsf{SRT}^2_{<\infty}$ and that, for $k>l$, $\mathsf{RT}^1_k\not\le_{\rm sc}\mathsf{SRT}^2_l$. The tree labeling arguments of this paper elaborate on the previous methods and make use of the following definitions. 

\begin{definition} $~$
\begin{itemize} 
	\item If $\alpha$ is a nonempty string of natural numbers, then $\alpha^{\#}$ denotes $\alpha\upharpoonright |\alpha|-1$. 
	\item If $\alpha$ and $\beta$ are strings of natural numbers, then $\alpha * \beta$ denotes $\alpha$ concatenated by $\beta$. 
	\item If $\alpha$ is a string of natural numbers and $x\in\omega$, then $\alpha * x$ denotes $\alpha * \langle x \rangle$.
	\item If $\alpha,\alpha*x\in T$ for some tree $T$, we say that $\alpha*x$ is a \textit{successor} of $\alpha$ in $T$. 
	\item If $A,B\subseteq\omega$, we write $A < B$ if every element of $A$ is less than every element of $B$. 
\end{itemize} 
\end{definition}

We now proceed define \textit{tree labeling with two labels}. 

\begin{definition} 
Let $H\subset\omega$ be finite and $I\subset\omega$ be infinite, with $H < I$. Let $\Gamma$ be a Turing functional. Finally, let $k\in\omega$. We define the tree $T(k,\Gamma,H,I)\subseteq I^{<\omega}$ by $\emptyset\in T(k,\Gamma,H,I)$ and for a nonempty string $\alpha$, $\alpha\in T(k,\Gamma,H,I)$ if $\alpha\in I^{<\omega}$ is increasing and there are no finite $F_L,F_R\subseteq\ran(\alpha^{\#})$ and no $b > a \ge k$ such that 
$$\Gamma^{H\cup (F_L\oplus F_R)}(a)\downarrow = \Gamma^{H\cup (F_L\oplus F_R)}(b)\downarrow = 1.$$
\end{definition}

Here the finite sets $F_L,F_R$ represent additions to be made to the left and right columns, respectively, of a $p$-homogeneous set in our later constructions while $H$ represents the initial segment of a $p$-homogeneous set built by some stage of a construction. 

\begin{remark}
$T=T(k,\Gamma,H,I)$ has the following three properties: 
\begin{itemize} 
	\item If $T$ is not well-founded and $P$ is any infinite path through $T$, then $\ran(P)\subseteq I$ is infinite and for all $F_L,F_R\subseteq\ran(P)$ and all $b > a \ge k$, 
	$$\Gamma^{H\cup(F_L\oplus F_R)}(a) \simeq \Gamma^{H\cup(F_L\oplus F_R)}(b) \simeq 0.$$ 
	\item If $\alpha\in T$, then if $\alpha$ is not terminal and $\ran(\alpha) < x\in I$, $\alpha * x\in T$. 
	\item If $\alpha\in T$ is terminal, then 
	$$\Gamma^{H\cup (F_L\oplus F_R)}(a)\downarrow = \Gamma^{H\cup (F_L\oplus F_R)}(b)\downarrow = 1$$
	for some $F_L,F_R\subseteq\ran(\alpha)$ and some $b > a \ge k$. 
\end{itemize} 
\end{remark}


\begin{definition} 
When $T = T(k,\Gamma,H,I)$ is well-founded, we \textit{label} the nodes of $T$ recursively, starting at the terminal nodes. Each node is labeled with an ordered pair whose elements may be natural numbers or the symbol $\infty$. 
\begin{itemize} 
	\item If $\alpha\in T$ is terminal, we label $\alpha$ with the least pair $\langle a,b \rangle$ of elements $b > a \ge k$ such that 
	$$\Gamma^{H\cup (F_L\oplus F_R)}(a)\downarrow = \Gamma^{H\cup (F_L\oplus F_R)}(b)\downarrow = 1$$ 
	for some $F_L,F_R\subseteq\ran(\alpha^{\#})$. 
	
	\item If $\alpha\in T$ is not terminal, then we determine the label of $\alpha$ starting with the second element as follows: 
	\begin{itemize} 
		\item If there is any $b\in\omega$ such that infinitely many of the successors of $\alpha$ have labels with second element $b$, then we let the least such $b$ be the second element of the label of $\alpha$. Otherwise, we let $\infty$ be the second element of the label of $\alpha$. 
		\item Now suppose the second element $b$ of the label of $\alpha$ has been determined already; we will determine its first element $a$ according as $b\in\omega$ or $b=\infty$. If $b=\infty$, then we let $a$ be the least finite number appearing as the first element of the label of infinitely many successors of $\alpha$, or else if there is no such finite number we let $a = \infty$. If $b\in\omega$, we restrict our attention to just those successors of $\alpha$ whose labels' second element is $b$, and let $a$ be the least number appearing as the first element of the label of infinitely many successors of $\alpha$. Observe that $a < b$. 
	\end{itemize} 
\end{itemize}
Note that no label which has the symbol $\infty$ as its first element has a finite number as its second. 
\end{definition}

\begin{definition} 
Suppose $T = T(k,\Gamma,H,I)$ is well-founded. Then the \textit{labeled subtree} $T^L = T^L(k,\Gamma,H,I)$ of $T$ is obtained from $T$ as follows. First, the root node of $T$ (namely $\emptyset$) is added to $T^L$. Now suppose we have added to $T^L$ some non-terminal node $\alpha$ of $T$. We then add to $T^L$ some of the successors of $\alpha$, thus: 
\begin{itemize} 
	\item If $\alpha$ has label $\langle a,b \rangle\in\omega^2$ in $T$, then we add to $T^L$ all those successors of $\alpha$ with the same label. 
	\item If $\alpha$ has label $\langle a,\infty \rangle$ for some $a\in\omega$ in $T$, then if infinitely many successors of $\alpha$ have label $\langle a,b \rangle\in\omega^2$, then for each $b\in\omega$ such that $\langle a,b \rangle$ appears as the label of a successor of $\alpha$, we select the least $x$ such that $\alpha * x$ has that label and add $\alpha * x$ to $T^L$; and if on the other hand cofinitely many successors of $\alpha$ have the same label (namely $\langle a, \infty \rangle$), we add all such successors to $T^L$.
	\item Otherwise, if $\alpha$ has label $\langle \infty,\infty \rangle$, then if infinitely many successors of $\alpha$ have label $\langle a,b \rangle\in\omega^2$, then for each such pair that appears as the label of a successor of $\alpha$, we select the least $x$ such that $\alpha * x$ has that label and add $\alpha * x$ to $T^L$; and if cofinitely many successors of $\alpha$ have label $\langle a,\infty$ for $a\in\omega$, then for each $a\in\omega$ such that $\langle a,\infty \rangle$ appears as the label of a successor of $\alpha$, we select the least $x$ such that $\alpha * x$ has that label and add $\alpha * x$ to $T^L$; and otherwise if cofinitely many successors of $\alpha$ have label $\langle \infty,\infty \rangle$, we add those successors to $T^L$. 
\end{itemize} 
All the nodes of $T^L$ retain the labels they had as nodes of $T$. Note that every node terminal in $T$ is also terminal in $T^L$, and that every non-terminal node in both $T$ and $T^L$ has infinitely many successors. 
\end{definition} 

\begin{definition} 
A node $\alpha\in T^L$ is called a \textit{transition node} if the symbol $\infty$ appears in the label of $\alpha$, and appears strictly fewer times in the label of each successor of $\alpha$. 
\end{definition}


\section{First Main Theorem}


\begin{theorem}
{\it There exists a stable 2-coloring of pairs $f:[\omega]^2\to 2$ and a family $Y$ of infinite sets such that no $(f\oplus P)$-computable set is homogeneous for $f$ for any $P\in Y$, and every stable 2-coloring of pairs $f':[\omega]^2\to 2$ computable from $f$ has either an $(f\oplus P)$-computable $p$-homogeneous set for some $P\in Y$, or if not then some $p$-homogeneous set which does not compute a set homogeneous for $f$. }
\end{theorem}


Before proving Theorem 3.1, we observe how the First Main Theorem is a direct consequence. Let $f$ and $Y$ be as in the statement of Theorem 3.1, let $\Phi,\Psi$ be any Turing functionals, and suppose that $\Phi^f$ is a stable 2-coloring of pairs. Let $H$ be set $p$-homogeneous for $\Phi^f$. Then by Theorem 3.1, $\Psi^H$ is not homogeneous for $f$. Thus $\mathsf{SRT}^2_2\not\le_{\rm sc}\mathsf{SPT}^2_2$. 

We define the following notion of forcing which we shall need to prove each of the main theorems. 

\begin{definition}
Let $\mathbb{C}$ denote the following notion of forcing. A condition is an ordered triple $p=\langle \sigma^p,l^p,|p| \rangle$ where $|p|\in\omega$, $\sigma^p:\left[|p|\right]^2\to 2$, $l^p:|p|\to 2\times\omega$, and $l^p(x) = \langle i,z \rangle$ implies that if $\sigma^p(x,y)$ is defined and $y\ge z$ then $\sigma^p(x,y) = i$. 
\end{definition}

From any sufficiently generic filter $\mathcal{G}$ for $\mathbb{C}$ we obtain a stable 2-coloring of pairs $f = \bigcup_{p\in \mathcal{G}}\sigma^p:[\omega]^2\to 2$ together with a function $l = \bigcup_{p\in\mathcal{G}}l^p:\omega\to\omega\times 2$ such that for each $x\in\omega$, $\lim_u f(x,u) = (l(x))_0$. 

\begin{lemma} 
If $p_0\ge p_1\ge p_2\ge \cdots$ is a sequence of $\mathbb{C}$-conditions which is 3-generic relative to some set $P\subseteq\omega$, and if $f = \bigcup_s \sigma^{p_s}:[\omega]^2\to 2$, then $f\oplus P$ does not compute a homogeneous set for $f$. 
\end{lemma}

\begin{proof}[Proof of Lemma 3.3]Fix $P\subseteq\omega$ such that $p_0\ge p_1\ge p_2\ge\cdots$ is 3-generic with respect to $P$ and fix a Turing functional $\Gamma$. Let $W_{P,\Gamma}$ be the set of all conditions $p$ which force one of the following two statements:

\begin{center}\parbox{5.5in}{
(a) $\Gamma^{f\oplus P}$ does not define an infinite set; \\[3pt] 
(b) there are $x,y\in\omega$ with $\Gamma^{f\oplus P}(x)\downarrow = \Gamma^{f\oplus P}(y)\downarrow = 1$ and $\lim_u f(x,u)\ne\lim_u f(y,u)$. 
}\end{center}

$W_{P,\Gamma}$ is $\Sigma^0_3$-definable in $P$, and we claim that $W_{P,\Gamma}$ is dense in $\mathbb{C}$. To see this, let $p$ be any condition none of whose extensions force (a), and suppose by way of contradiction that no extension of $p$ forces (b). This means that for every $x,y\in\omega$, if $q\le p$ then $q$ does not force both that $\lim_u f(x,u)\ne \lim_u f(y,u)$ and that $\Gamma^{f\oplus P}(x)\downarrow = \Gamma^{f\oplus P}(y)\downarrow = 1$. Then in particular there are no $x,y\ge |p|$, $\tau$ extending $\sigma^p$, and $L$ extending $l^p$ such that

\begin{center}\parbox{2.5in}{
(1) $\tau$ respects $L$; \\[3pt] 
(2) $\Gamma^{\tau \oplus P \upharpoonright |p|}(x)\downarrow = \Gamma^{\tau \oplus P \upharpoonright |p|}(y)\downarrow=1$; and \\[3pt] 
(3) $(L(x))_0\ne (L(y))_0$. 
}\end{center}

Now if there is no such $x,y,\tau,L$ satisfying (1) and (2), then $p$ forces that $\Gamma^{f\oplus P}$ does not define an infinite set, which is a contradiction. Therefore there are $x,y,\tau,L$ satisfying (1) and (2). But (3) is independent of (1) and (2); given $x,y,\tau$ satisfying (1) and (2) we may find $L$ extending $l^p$ and compatible with $\tau$ such that $(L(x))_0\ne (L(y))_0$. We conclude that $W$ is dense in $\mathbb{C}$.
\end{proof}


\begin{proof}[Proof of Theorem 3.1]
We build
\begin{itemize} 
\item a sequence of $\mathbb{C}$-conditions $p_0\ge p_1\ge p_2\ge\cdots$ with $\lim_s |p_s| = \infty$; 
\item sequences of finite sets (initial segments of $p$-homogeneous sets) $H^\Phi_{j,0}\subseteq H^\Phi_{j,1}\subseteq H^\Phi_{j,2}\subseteq\cdots$ for each Turing functional $\Phi$ and each $j<2$; 
\item a sequence of infinite sets (reservoirs) $I_0\supseteq I_1\supseteq I_2\supseteq\cdots$ with $H^\Phi_{j,s} < I_s$ for each $\Phi,j,s$; 
\item a sequence of finite families $Y_0\subseteq Y_1\subseteq Y_2\subseteq\cdots$ of infinite subsets of $\omega$; 
\end{itemize} 
and we define $f = \bigcup_s \sigma^{p_s}$, $H^\Phi_j = \bigcup_s H^\Phi_{j,s}$ for each $j<2$, and $Y = \bigcup_s Y_s$. The construction will ensure the following requirements, for all $i\in\omega$ and Turing functionals $\Phi,\Gamma,\Delta$: 

\hspace*{9mm}\begin{center}\begin{tabular}{lcl} 
$\mathcal{P}_i$ & : & the sequence $p_0\ge p_1\ge p_2\ge \cdots$ is 3-generic relative to each $P\in Y_i$; \\ 
& & \\
$\mathcal{Q}_{\Phi,i}$ & : & if $\Phi^f$ is a stable 2-coloring of pairs, it either has an $(f\oplus P)$-computable \\ 
& &  $p$-homogeneous set for some $P\in Y$ or else both $H^\Phi_0$ and $H^\Phi_1$ are infinite \\ & & in both columns (in other words, given our encoding of $p$-homogeneous \\ & & sets, contain infinitely many even and infinitely many odd numbers); \\ 
& & \\
$\mathcal{R}^\Phi_{\Gamma,\Delta}$ & : & if $\Phi^f$ is a stable 2-coloring of pairs, it either has an $(f\oplus P)$-computable \\ 
& & $p$-homogeneous set for some $P\in Y$; or else if $\Gamma^{H^\Phi_0}$ defines an infinite set \\ 
& & then this set is not homogeneous for $f$; or else if $\Delta^{H^\Phi_1}$ defines an infinite \\ 
& & set then this set is not homogeneous for $f$. 
\end{tabular}\end{center}  $~$

By way of explaining the $\mathcal{Q}$ and $\mathcal{R}$ requirements, note that by Lemma 3.3, if there is any sequence satisfying the $\mathcal{P}$ requirements and such that for some $P\in Y$ $f\oplus P$ computes a set $p$-homogeneous for $\Phi^f$, there will then be a set which is $p$-homogeneous for $\Phi^f$ but which computes no set homogeneous for $f$. 

\underline{\it Construction.} 

Devote infinitely many stages $s\in\omega$ to each requirement. Let $p_0$ be any condition with $|p_0| = 0$. For each $\Phi$ let $H^\Phi_{0,0} = H^\Phi_{1,0} = \emptyset$, and let $I_0 = \omega$ and $Y_0 = \emptyset$. At stage $s+1$ assume by way of induction that we have $p_s$, $H^\Phi_{j,s}$ for $j<2$ and all $\Phi$, $I_s$, and $Y_s$ and assume that if $H^\Phi_{j,s}$ is nonempty for some $j$ and $\Phi$, then $p_s$ forces that $\Phi^f$ is a stable coloring of pairs and that $\Phi^f(x,y) = j$ whenever $2x,2y+1\in H^\Phi_{j,s}$ or when $y\in I_s$ and either $2x\in H^\Phi_{j,s}$ or $2x+1\in H^\Phi_{j,s}$. At the end of a stage any of $p_{s+1}$, $H^\Phi_{j,s+1}$, $I_{s+1}$, or $Y_{s+1}$ not yet defined should be taken to be identical to $p_s$, $H^\Phi_{j,s}$, $I_{s}$, or $Y_s$, respectively. 


\subsection*{$\mathcal{P}$ requirements.} Suppose $s$ is dedicated to requirement $\mathcal{P}_i$ for some $i < s$ and that it is the $\langle n,m \rangle^{\rm th}$ such stage. If $n > |Y_i|$ do nothing. Otherwise, let $P$ be the $n^{\rm th}$ member of the family $Y_i$ in some fixed enumeration and let $W$ be the $m^{\rm th}$ $\Sigma^0_3(P)$ set in some fixed enumeration. If $p_s$ has an extension $q$ in $W$, fix $q$ and let $p_{s+1} = q$, so that $p_0\ge p_1\ge \cdots \ge p_s\ge p_{s+1}\ge \cdots$ meets $W$. Otherwise, do nothing, and $p_0\ge p_1\ge \cdots \ge p_s\ge p_{s+1}\ge \cdots$ avoids $W$. 


\subsection*{$\mathcal{Q}$ requirements.} Suppose $s$ is dedicated to $\mathcal{Q}_{\Phi,i}$. Without loss of generality we assume that $p_s$ decides whether or not $\Phi^f$ is a stable 2-coloring of pairs. If $p_s$ forces that $\Phi^f$ is not such a coloring, do nothing. Otherwise, we consider two cases. 
\begin{itemize} 
\item If for some $j < 2$ and $k\in\omega$ there is no extension of $p_s$ which forces that $\lim_u\Phi^f(x,u) = j$ for some $x\ge k$ in $I_s$, then $P = \{x\in I_s:x\ge k\}$ is limit homogeneous for $\Phi^f$ with color $1-j$, and so $(f\oplus P)$ computes a set $p$-homogeneous for $\Phi^f$. We set $Y_{s+1} = Y_s\cup\{P\}$. This satisfies the requirement (since, as remarked earlier, this means there is a set which is $p$-homogeneous for $\Phi^f$ but which computes no set homogeneous for $f$). 

\item If no such $j,k$ exist, then there are numbers $x_{00},x_{01},x_{10},x_{11}\in I_s$ and an extension of $p_s$ forcing that $H^\Phi_{j,s}\cup\{2x_{j0},2x_{j1}+1\}$ is finite $p$-homogeneous for $\Phi^f$ and $\lim_u\Phi^f(x_{ji},u) = j$ for each $i,j<2$. In this case let $p_{s+1}$ be such an extension of $p_s$, let $H^\Phi_{j,s+1} = H^\Phi_{j,s}\cup\{2x_{j0},2x_{j1}+1\}$, and let $I_{s+1} = \{x:\,m<x\in I_s\}$ where $m$ is greater than the stabilization points under $\Phi^f$ of every element of $H^\Phi_{0,s+1}\cup H^\Phi_{1,s+1}$. Observe that both columns have been extended by one element. 
\end{itemize}


\subsection*{$\mathcal{R}$ requirements.} Suppose $s$ is dedicated to $\mathcal{R}^{\Phi}_{\Gamma,\Delta}$ and assume that $p_s$ forces that $\Phi^f$ is a stable coloring of pairs. The goal of this requirement is to extend by some finite set either the initial segment of $H^\Phi_0$ or the initial segment of $H^\Phi_1$ that we have constructed so far, subject to the following condition: if we extend $H^\Phi_0$, then $\Gamma$ does not compute a homogeneous set for $f$ from any further extension of $H^\Phi_0$; and if on the other hand we extend $H^\Phi_1$, then $\Delta$ does not compute a homogeneous set for $f$ from any further extension of $H^\Phi_1$. When we have so extended one of $H^\Phi_0,H^\Phi_1$, we will say that we have successfully \textit{diagonalized} against such computations. 

We intend to accomplish this diagonalization in the following way. There will be two numbers, say $a$ and $b$, and two finite sets $F_L$ and $F_R$, all arising from a tree labeling construction, about which we know either that 
$$\Gamma^{H^\Phi_0\cup (F_L\oplus F_R)}(a)\downarrow = \Gamma^{H^\Phi_0\cup (F_L\oplus F_R)}(b)\downarrow = 1$$ 
or else that 
$$\Delta^{H^\Phi_1\cup (F_L\oplus F_R)}(a)\downarrow = \Delta^{H^\Phi_1\cup (F_L\oplus F_R)}(b)\downarrow = 1.$$ 
Let us suppose we know the first. In other words, we know that if we extend $H^\Phi_0$ by $F_L\oplus F_R$, then the set computed from $H^\Phi_0\cup(F_L\oplus F_R)$ by $\Gamma$ will contain $a$ and $b$. Thus to diagonalize---i.e. to ensure that the set so computed is not homogeneous for $f$---we will choose an extension $q$ of the condition $p_s$ such that $\sigma^q(a,b)$, $(l^q(a))_0$, and $(l^q(b))_0$ are not all equal. This will guarantee that no set containing both $a$ and $b$ is homogeneous for the coloring eventually obtained by extending $\sigma^q$. The tension of the proof arises from the fact that, while diagonalizing in this way, we also need $q$ to force that elements of $F_L$ and $F_R$ have the right limits under $\Phi^f$ so that $H^\Phi_0\cup(F_L\oplus F_R)$ can in fact be extended to a $p$-homogeneous set. 

The construction divides into two cases, \textbf{Case B} and \textbf{Case A}. In Case B we suppose there is an $i<2$, a condition $q\le p_s$, an infinite set $I\subseteq I_s$, and a set $Q = \{\langle x,a_x,b_x \rangle: x\in I,\,a_x,b_x\in\omega\}$ such that there is exactly one triple in $Q$ with first element $x$ for each $x\in I$, $a_x < b_x$, and such that infinitely many numbers $b_x\in\omega$ appear as third elements of triples in $Q$; and we suppose these have the following property: for any $x\in I$ and any extension $r$ of $q$ such that $\sigma^r(a_x,b_x)$, $(l^r(a_x))_0$, and $(l^r(b_x))_0$ are not all equal will force that $\lim_u \Phi^f(x,u) = i$. In Case A we suppose this is not so. The reason for choosing this way of dividing the cases will become apparent later. Case B may be regarded as the difficult case, where the real heavy lifting of the proof takes place; Case A may be regarded as the relatively easy case. 

\subsection*{Case A}

Let $T_0 = T(|p_s|,\Gamma,H^\Phi_{0,s},I_s)$. If $T_0$ is not well-founded then let $I_{s+1}$ be the range of an infinite path through $T_0$. Observe that in this case the requirement is satisfied. If $T_0$ is well-founded, then let $T^L_0$ be the labeled subtree of $T_0$. 

We now try to define two sequences, conditions 
$$p_s\ge q_0 \ge q_1 \ge q_2 \ge \cdots$$ 
and nodes of $T^L_0$ 
$$\emptyset = \alpha_0 \preceq \alpha_1 \preceq \alpha_2 \preceq \cdots$$ 
where for all $j\ge 0$, $\alpha_{j+1}$ is a successor of $\alpha_j$ and for all $j\ge 0$ the condition $q_j$ forces that 
$$\lim_u \Phi^f(x,u) = 0$$ 
for all $x\in\ran(\alpha_j)$. We begin the definition of these sequences as follows. 

\begin{itemize} 
\item If $\emptyset = \alpha_0$ has label $\langle a,b \rangle$ with $a,b\in\omega$, let $q_0$ be any extension of $p_s$ having $\sigma^{q_0}(a,b)$, $(l^{q_0}(a))_0$, and $(l^{q_0}(b))_0$ not all equal. In this case the diagonalization for the present requirement is now complete. 
\item Otherwise, let $q_0 = p_s$. 
\end{itemize}


We then proceed by induction until either the induction fails and we satisfy the requirement by adding a certain set $P$ to $Y$, or else the induction always succeeds at every non-terminal node. In the latter case, once we reach a terminal node we will be ready to diagonalize . 
Suppose we have defined $q_n$ and $\alpha_n$ and that the latter is not terminal in $T^L_0$. Recall that $q_n$ forces that there is some $m\in\omega$ such that $\Phi^f(x,y) = 0$ for $x\in\ran(\alpha_n)$ and $y\ge m$. Let $S = S(n)$ be the set of all successors $\alpha_n*x$ of $\alpha_n$ with $x\ge m$. The induction breaks into cases according as $\alpha_n$ is or is not a transition node. At the first suitable transition node, we set up to diagonalize. 

\subsection*{Case A.1.} If $\alpha_n$ is not a transition node, let $P = \{x:\alpha_n*x\in S\}$. We look for an $x^*\in P$ and an extension $q$ of $q_n$ which forces that 
$$\lim_u \Phi^f(x^*,u) = 0.$$ 
If we find such, we let $q_{n+1} = q$ and we let $\alpha_{n+1}$ be any $\beta\in S$ with $\beta(n) = x^*$; if we find no such, then $P$ is limit-homogeneous for $\Phi^f$ and thus $f\oplus P$ computes a $p$-homogeneous set for $\Phi^f$ (for example, we can thin $P$ computably in $f$ to a set $G$ homogeneous for $\Phi^f$; then $G\oplus G$ is $p$-homogeneous for $\Phi^f$). In this case we set $Y_{s+1} = Y_s\cup\{P\}$ and $p_{s+1} = q_n$, satisfying the requirement and ending stage $s$.

\subsection*{Case A.2.} If $\alpha_n$ is a transition node, the induction breaks into two cases. The header for each case gives a shorthand for the sort of transition being
discussed.
\subsection*{Case A.2.1.} $\langle \infty,\infty \rangle \to \langle a,\infty \rangle$ \\ 

If $\alpha_n$ has label $\langle \infty,\infty \rangle$ and every successor of $\alpha_n$ has a label in which the symbol $\infty$ appears exactly once, then we proceed as in the non-transition case. 
\subsection*{Case A.2.2.} $\langle \infty,\infty \rangle \to \langle a,b \rangle$ or $\langle a,\infty \rangle \to \langle a,b \rangle$ \\ 

If $\alpha_n$ has a label in which the symbol $\infty$ appears but every successor of $\alpha_n$ is labeled only with finite numbers, 
then let 
$$P = \{\langle x, a, b \rangle:\alpha_n*x\in S\mbox{ and }\alpha_n*x\mbox{ has label }\langle a,b \rangle\mbox{ and }b > |q_n|\}.$$  
We look for a tuple $\langle x^*,a^*,b^* \rangle\in P$ and an extension $q$ of $q_n$ which forces that 
$$\lim_u \Phi^f(x^*,u) = 0$$ 
and is such that $\sigma^q(a^*,b^*)$, $(l^q(a^*))_0$, and $(l^q(b^*))_0$ are not all equal. Observe that, by the assumption that we are not in Case B, we must find such a tuple and extension. For otherwise fix $Q$ to be a set of triples $\langle x, a_x, b_x\rangle$ such that for every $b$ with $\langle y, a, b \rangle\in P$ for some $y$ and $a$, there is precisely one triple $\langle x, a_x, b_x\rangle$ in $Q$ with $b_x = b$ and let $I = \{x:\exists a_x,b_x\,(\langle x,a_x,b_x \rangle\in Q)\}$. Then we are in Case B with $i=1$ and $q=q_n$. Thus, select such an extension $q$ and let $q_{n+1} = q$ and we let $\alpha_{n+1}$ be any $\beta\in S$ with $\beta(n) = x^*$ and label $\langle a^*, b^* \rangle$. As $\alpha_{n+1}$ is not a transition node, we return to Case A.1.

\subsection*{Case B}

Fix an $i<2$, a condition $q\le p_s$, an infinite set $I\subseteq I_s$, and a set $Q = \{\langle x,a_x,b_x \rangle: x\in I,\,a_x,b_x\in\omega\}$ such that there is exactly one triple in $Q$ with first element $x$ for each $x\in I$; such that infinitely many numbers $b_x\in\omega$ appear as third elements of triples in $Q$; and such that for any $x\in I$ and any extension $r$ of $q$ such that $\sigma^r(a_x,b_x)$, $(l^r(a_x))_0$, and $(l^r(b_x))_0$ are not all equal will force that $\lim_u \Phi^f(x,u) = i$. Suppose without loss of generality that $i=1$. In the tree labeling construction which follows, we will extend $q$ several times. As we gradually extend the forcing condition, for any given $x$ it might occur---and we may as well assume---that the limiting color of $a_x$ will be forced, and that the stabilization point will fall below $b_x$ so that the color of $(a_x,b_x)$ will be forced also. However, since there are infinitely many different numbers $b_x$, by restricting our attention to those $x$'s for which the limiting color of $b_x$ has not yet been forced, we will always have the option of choosing an extension $q$ as above, i.e. such that $\sigma^q(a_x,b_x)$, $(l^q(a_x))_0$, and $(l^q(b_x))_0$ are not all equal, and thus force that $\lim_u \Phi^f(x,u) = 1$. The fact that we can always accomplish this while making progress toward satisfying the present requirement underlies the remainder of the construction, and so we find it useful to define a term describing extensions $q_n$ which accomplish this. 

\begin{definition}
A forcing extension $q$ is said to \textit{press the button of $x$} or to \textit{press $b_x$} if $\sigma^q(a_x,b_x)$, $(l^q(a_x))_0$, and $(l^q(b_x))_0$ are not all equal.
\end{definition}

\noindent Thus (assuming as above that $i=1$) any forcing extension which presses the button of $x$ forces that $\lim_u \Phi^f(x,u)=1$.

For this case we need to modify our definition of tree labeling. We will use \textit{tree labeling with three labels}. The definition of $T(k,\Gamma,H,I)$ is for this case changed to the following: $\emptyset \in T(k,\Gamma,H,I)$ and for a nonempty string $\alpha$, $\alpha\in T(k,\Gamma,H,I)$ if $\alpha\in I^{<\omega}$ is increasing and there are no finite $F_L,F_R\subseteq\ran(\alpha^{\#})$ and no $c > b > a \ge k$ such that 
$$\Gamma^{H\cup(F_L\oplus F_R)}(a)\downarrow = \Gamma^{H\cup(F_L\oplus F_R)}(b)\downarrow = \Gamma^{H\cup(F_L\oplus F_R)}(c)\downarrow = 1.$$
The method for labeling the nodes of $T(k,\Gamma,H,I)$ extends the method from tree labeling with two labels in the natural way, as does the method for selecting the nodes of the labeled subtree $T^L(k,\Gamma,H,I)$. Remark 2.3 applies \textit{mutatis mutandis}. Additionally, we revise the definition of a transition node: a node $\alpha\in T^L$ is called a transition node if the symbol $\infty$ appears in the label of $\alpha$, and appears strictly fewer times \textit{and no more than once} in the label of each successor of $\alpha$. 

Now let $T_1 = T(|q|,\Delta,H^{\Phi}_{1,s},I_s)$. If $T_1$ is not well founded then let $I_{s+1}$ be the range of an infinite path through $T_1$. Observe that in this case the requirement is satisfied. If $T_1$ is well founded, let $T^L_1$ be the labeled subtree of $T_1$.

We now try to define two sequences, conditions 
$$q \ge q_0 \ge q_1 \ge q_2 \ge \cdots$$ 
and nodes of $T^L_1$ 
$$\emptyset = \alpha_0 \preceq \alpha_1 \preceq \alpha_2 \preceq \cdots$$ 
where for all $j\ge 0$, $\alpha_{j+1}$ is a successor of $\alpha_j$ and for all $j\ge 0$ the condition $q_j$ forces that 
$$\lim_u \Phi^f(x,u) = 1$$ 
for all $x\in\ran(\alpha_j)$. 
We begin the definition of these sequences as follows. 

\begin{itemize} 
\item If $\emptyset = \alpha_0$ has label $\langle a,b,c \rangle$ with $a,b,c\in\omega$, let $q_0$ be any extension of $q$ having $\sigma^{q_0}(a,b)$, $(l^{q_0}(a))_0$, and $(l^{q_0}(b))_0$ not all equal. 
\item If $\emptyset = \alpha_0$ has label $\langle a,b,\infty \rangle$ with $a,b\in\omega$, let $q_0$ be any extension of $q$ having $\sigma^{q_0}(a,b)$, $(l^{q_0}(a))_0$, and $(l^{q_0}(b))_0$ not all equal. 
\item Otherwise, let $q_0 = q$. 
\end{itemize}

We then proceed by induction. Suppose we have defined $q_n$ and $\alpha_n$ and that the latter is not terminal in $T^L_1$. Recall that $q_n$ forces that there is some $m\in\omega$ such that $\Phi^f(x,y) = 1$ for $x\in\ran(\alpha_n)$ and $y\ge m$. Let $S$ be the set of all successors $\alpha_n*x$ of $\alpha_n$ with $x\ge m$ and $b_{x} > |q_n|$. The induction breaks into cases according as $\alpha_n$ is or is not a transition node, with diagonalization occurring at transition nodes. 

\subsection*{Case B.1.} If $\alpha_n$ is not a transition node, let $P = \{x:\alpha_n*x\in S\}$. We choose any $x^*\in P$ and any extension $r$ of $q_n$ which presses $b_{x^*}$, and let $q_{n+1} = r$ and $\alpha_{n+1} = \gamma$ for any $\gamma\in S$ having $\gamma(n) = x^*$.

\subsection*{Case B.2.} If $\alpha_n$ is a transition node, the induction breaks into four subcases. Each case is detailed in a separate bullet point below, and the header under each bullet point gives a shorthand for the sort of transition being discussed. The real work of the construction, making full use of the triple labeling of $T^L_1$, takes place in subcases 3 and 4.

\subsection*{Case B.2.1.} $\langle \infty,\infty,\infty \rangle \to \langle a,b,\infty \rangle \mbox{ or } \langle a,b,c \rangle$ \\ 

If $\alpha_n$ has label $\langle \infty, \infty, \infty \rangle$ and every successor of $\alpha_n$ has a label in which the symbol $\infty$ appears at most once, then we let 
$$P = \{x:\alpha_n*x\in S\land \forall j<n\,(a,b > b_{\alpha_n(j)})\land a,b > |q_n|)\},$$ 
where $a,b$ here denote the first two entries in the label of $\alpha*x$ and $b_{\alpha_n(j)}$ denotes the button of $\alpha_n(j)$. 
We choose any $x^*\in P$ and any extension $r$ of $q_n$ such that, if the label of $x^*$ is $\langle a^*, b^*, c^* \rangle$ with $a^*,b^*\in\omega$ and $c^*\in\omega\cup\{\infty\}$, then $r$ presses $b_{x^*}$ and $\sigma^r(a^*,b^*)$, $(l^r(a^*))_0$, and $(l^r(b^*))_0$ are not all equal. Recall the definition of $S$, whereby we know that $b_x > |q_n|$ for each $x\in P$. Thus for all $x\in P$, $l^{q_n}$ does not commit us to any particular limiting color for $b_x$. This is why we are free to press the button of $x^*$. Let $q_{n+1} = r$ and $\alpha_{n+1} = \gamma$ for any $\gamma\in S$ having $\gamma(n) = x^*$ with label $\langle a^*,b^*,c^*\rangle$. As $\alpha_{n+1}$ is not a transition node, we now return to Case B.1.
\subsection*{Case B.2.2.} $\langle a,\infty,\infty \rangle \to \langle a,b,c \rangle$ \\ 

If $\alpha_n$ has label $\langle a^*, \infty, \infty \rangle$ and every successor of $\alpha_n$ has a label in which only finite numbers appear, then we let 
$$P = \{x:\alpha_n*x\in S\land \forall j<n\,(b,c > b_{\alpha_n(j)})\land b,c > |q_n|)\},$$ 
where $b,c$ here denote the second and third entries in the label of $\alpha*x$ and $b_{\alpha_n(j)}$ denotes the button of $\alpha_n(j)$. We choose any $x^*\in P$ and any extension $r$ of $q_n$ that presses $b_{x^*}$ and is such that, if the label of $x^*$ is $\langle a^*, b^*, c^*\rangle$ with $a^*,b^*,c^*\in\omega$, then $\sigma^r(b^*,c^*)$, $(l^r(b^*))_0$, and $(l^r(c^*))_0$ are not all equal. Let $q_{n+1} = r$ and $\alpha_{n+1} = \gamma$ for any $\gamma\in S$ having $\gamma(n) = x^*$ with label $\langle a^*,b^*,c^*\rangle$. As $\alpha_{n+1}$ is not a transition node, we now return to Case B.1.
\subsection*{Case B.2.3.} $\langle a,\infty,\infty \rangle \to \langle a,b,\infty \rangle$ \\ 

%


For this case we first define by induction a function $\mathsf{Sort}$ on nodes $\alpha\in T_1^L$ which records information about how elements that go into $F_L\oplus F_R$ are sorted between $F_L$ and $F_R$. We define $\mathsf{Sort}$ as follows: 
\begin{itemize} 
\item If $\alpha$ is terminal and we have labeled $\alpha$ with the triple $\langle a,b,c\rangle$ because 
$$\Delta^{H^\Phi_{1,s}\cup (F_L\oplus F_R)}(a)\downarrow = \Delta^{H^\Phi_{1,s}\cup (F_L\oplus F_R)}(b)\downarrow = \Delta^{H^\Phi_{1,s}\cup (F_L\oplus F_R)}(c)\downarrow = 1,$$ 
then we let $\mathsf{Sort}(\alpha) = \{F_L\oplus F_R\}$. 
\item If $\alpha$ is not terminal, then we define 
$$\mathsf{Sort}(\alpha) = \bigcup\left\{S:\,S=\mathsf{Sort}(\beta)\mbox{ for infinitely many 1-extensions $\beta$ of $\alpha$}\right\}.$$
\end{itemize}

\noindent Now we define a convenient term. If $\alpha\in T = T(k,\Gamma,H,I)$ is not terminal and if $x,y\in\ran(\alpha)$, we say that $x$ and $y$ \textit{share a column} if there is $F_L\oplus F_R\in\mathsf{Sort}(\alpha)$ such that one of the following is true: 
\begin{itemize} 
	\item $x\in F_L$ and $y\in F_L$; 
	\item $x\in F_R$ and $y\in F_R$; 
	\item $x\notin F_L\cup F_R$ or $y\notin F_L\cup F_R$. 
\end{itemize}

\noindent Informally, $x$ and $y$ share a column if, in the course of the construction, we can avoid building a $p$-homogeneous set in which $x$ and $y$ appear in different columns. We now resume the construction.


If $\alpha_n$ has label $\langle a^*, \infty, \infty \rangle$ and successor of $\alpha_n$ has a label in which the symbol $\infty$ appears exactly once, then we let 
$$P = \{y:\alpha_n*y\in S\land \forall j<n\,(b > b_{\alpha_n(j)})\land b > |q_n|)\},$$ 
where $b$ here denotes the second entry in the label of $\alpha_n*y$ and $b_{\alpha_n(j)}$ denotes the button of $\alpha_n(j)$. Whether there is much work to be done in this case depends on whether the first node to have label $\langle a^*,\infty,\infty \rangle$ was or was not the root node of $T^L_1$. Formally, suppose $k$ is the least index such that $\alpha_k$ has a label in which the symbol $\infty$ appears exactly twice. 
\begin{itemize}
\item If, on the one hand, $\alpha_k = \alpha_0 = \emptyset$, then we choose any $y^*\in P$ and any extension $r$ of $q_n$ which presses $b_{y^*}$ and is such that, if the label of $y^*$ is $\langle a^*,b^*,\infty \rangle$, then $\sigma^r(a^*,b^*)$, $(l^r(a^*))_0$, and $(l^r(b^*))_0$ are not all equal. In this case let $q_{n+1} = r$ and $\alpha_{n+1} = \gamma$ for any $\gamma\in S$ having $\gamma(n) = x^*$ with label $\langle a^*,b^*,\infty\rangle$. 

\item If, on the other hand, $k > 0$ and $\alpha_k\ne\emptyset$, then there is more work to do. Suppose that $\alpha_k(k-1) = x^*$. Let $P'\subseteq P$ contain precisely the elements $y$ of $P$ which share a column with $x^*$. This is where we begin to use the triple labels of $T^L_1$. Either $P'\ne\emptyset$ and we complete the diagonalization in this case, or else $P'=\emptyset$ and we wait until the next case to diagonalize, but we are guaranteed to succeed when we attempt in the next case to find pairs of elements which share a column.
\end{itemize}
The important idea here is intuitively as follows. Either we may choose $y^*$ from $P$ to share with $x^*$ a column of the $p$-homogeneous set, or else $x^*$ and $y^*$ do not share a column and then $z^*$ may be chosen at the next transition node to share a column with one of $x^*$ or $y^*$. In either case, the color that $\Phi^f$ assigns to the pair of numbers which share a column can safely be changed without disturbing the construction. 


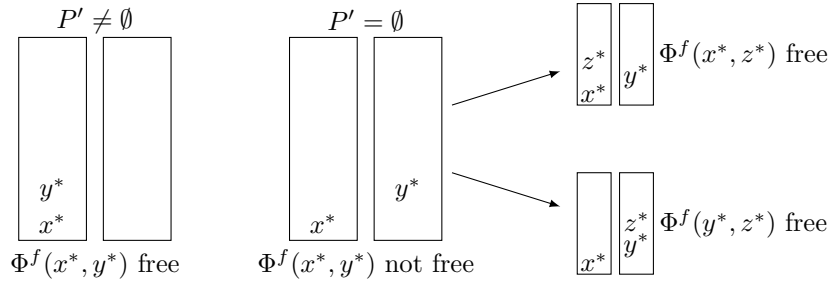
\begin{center}\begin{tikzpicture}[scale=0.9]
\node at (1.125,3.25) {$P'\ne\emptyset$};
\draw (0,0) rectangle (1,3); \node at (0.5,0.25) {$x^*$};
\draw (1.25,0) rectangle (2.25,3); \node at (0.5,0.75) {$y^*$};
\node at (5.125,3.25) {$P'=\emptyset$};
\draw (4,0) rectangle (5,3); \node at (5.75,0.75) {$y^*$};
\draw (5.25,0) rectangle (6.25,3); \node at (4.5,0.25) {$x^*$};
\draw[-latex] (6.4,1) -- (8,0.5);
\draw[-latex] (6.4,2) -- (8,2.5);
	\begin{scope}[xshift=6.25cm, yshift=2cm, scale=0.5]
	\draw (4,0) rectangle (5,3); \node at (5.75,0.85) {$y^*$};
	\draw (5.25,0) rectangle (6.25,3); \node at (4.5,0.35) {$x^*$};
	\node at (4.5,1.35) {$z^*$};
	\end{scope}
\begin{scope}[xshift=6.25cm, yshift=-0.5cm, scale=0.5]
\draw (4,0) rectangle (5,3); \node at (5.75,0.85) {$y^*$};
\draw (5.25,0) rectangle (6.25,3); \node at (4.5,0.35) {$x^*$};
\node at (5.75,1.55) {$z^*$};
\end{scope}
\node at (1.125,-0.35) {$\Phi^f(x^*,y^*)$ free};
\node at (5.125,-0.35) {$\Phi^f(x^*,y^*)$ not free};
\node at (10.7,0.25) {$\Phi^f(y^*,z^*)$ free};
\node at (10.7,2.75) {$\Phi^f(x^*,z^*)$ free};
\end{tikzpicture}\end{center}


If $P'$ is nonempty, then we choose any $y^*\in P'$ with label $\langle a^*,b^*,\infty\rangle$ and a condition $r$ which extends $q_n$ except possibly having $(l^r(a^*))_1\ne (l^{q_n}(a^*))_1$ if the latter is defined; and we choose $y^*$ and $r$ such that $r$ presses $b_{y^*}$ and $\sigma^r(a^*,b^*)$, $(l^r(a^*))_0$, and $(l^r(b^*))_0$ are not all equal; and we let $q_{n+1} = r$ and $\alpha_{n+1} = \gamma$ for any $\gamma\in S$ having $\gamma(n) = y^*$ with label $\langle a^*,b^*,\infty\rangle$. Finally, if $\sigma$ is any non-terminal extension of $\alpha_{n+1}$ in $T^L_1$ such that $x^*$ and $y^*$ do not share a column, then we delete from $T^L_1$ $\sigma$ and all of its extensions. As $\alpha_{n+1}$ is not a transition node, we now return to Case B.1.

Such $y^*$ and $r$ exist in this case for the following reason. Observe that for any $y^*\in P'$ with label $\langle a^*,b^*,\infty\rangle$, $a^*\ne b_{\alpha_n(j)}$ for $j\ne k-1$. Since $b^* > |q_n|$, $\sigma^{q_n}(a^*,b^*)$ is not yet defined, so we may choose $r$ extending $q_n$ with $\sigma^{r}(a^*,b^*)\ne (l^r(a^*))_0$---unless $l^{q_n}(a^*)$ is defined and $(l^{q_n}(a^*))_1\le b^*$. In this latter case we let $(l^r(a^*))_0 = (l^{q_n}(a^*))_0$ but choose $(l^r(a^*))_1 > b^*$. Note that in this case $r$ is not an extension of $q_n$ but that $r$ does extend $p_s$. From here we extend $r$ rather than $q_n$. 

To conclude the argument in this case, it remains to observe that changing the stabilization point of $a^*$ as above does not injure our construction in any way. For the colors of pairs of elements of $\alpha_n$ are forced by facts about $\sigma^{q_n}$ alone. Furthermore, the fact that $\Phi^f(u,y^*) = 1$ for all $u\in\ran(\alpha_n)$, $u\ne x^*$ with parity unequal to that of $y^*$ is settled by facts about the stabilization points of such numbers $u$, and these are unaffected by altering $(l^{q_n}(a^*))_1$ since $a^*$ is not the button of any such $u$. This means that choosing $(l^r(a^*))_1\ne(l^{q_n}(a^*))_1$ can at most interfere with the stabilization point of $x^*$ and keep us from forcing $\Phi^f(x^*,y^*)=1$. But note that this does not matter for purposes of the $p$-homogeneous set we are building, 
since $x^*$ and $y^*$ will not appear in different columns of the $p$-homogeneous set we are building and so their mutual color is irrelevant. 

\setlength{\leftskip}{0pt}
\setlength{\rightskip}{0pt}

If on the other hand $P'$ is empty, then we proceed as in the non-transition case and will diagonalize in Case B.2.4 instead. 
\newpage
\subsection*{Case B.2.4.} $\langle a,b,\infty \rangle \to \langle a,b,c \rangle$ \\ 

If $\alpha_n$ has label $\langle a^*,b^*,\infty\rangle$ and every successor of $\alpha_n$ has a label in which only finite numbers occur, and if we failed to diagonalize at an earlier node in the previous case, then we proceed as follows; otherwise we proceed as in the non-transition case. Let 
$$P = \{z:\alpha_n*z\in S\land \forall j<n\,(c > b_{\alpha_n(j)})\land c > |q_n|)\},$$ 
where $c$ here denotes the third entry in the label of $\alpha_n*z$. Suppose $k$ is the least index such that $\alpha_k$ has a label in which the symbol $\infty$ appears exactly twice and that $l$ is the least index such that $\alpha_l$ has a label in which the symbol $\infty$ appears exactly once. Suppose that $\alpha_k(k-1) = x^*$ and that $\alpha_l(l-1) = y^*$. Let $P'\subseteq P$ contain precisely those elements of $P$ which share a column with $x^*$ and $P''\subseteq P$ contain precisely those elements of $P$ which share a column with $y^*$. At least one of $P',P''$ must be nonempty; without loss of generality we assume that $P''$ is nonempty. Then we choose $z^*\in P''$ with label $\langle a^*,b^*,c^*\rangle$ and a condition $r$ which extends $q_n$ except possibly having $(l^r(b^*))_1\ne (l^{q_n}(b^*))_1$ if the latter is defined; and we choose $z^*$ and $r$ such that $\sigma^r(b^*,c^*)$, $(l^r(b^*))_0$, and $(l^r(c^*))_0$ are not all equal and $r$ presses $b_{z^*}$. We let $q_{n+1} = r$ and $\alpha_{n+1} = \gamma$ for any $\gamma\in S$ having $\gamma(n) = z^*$ with label $\langle a^*,b^*,c^*\rangle$. Such $z^*$ and $r$ exist by the same reasoning given in the previous case. Finally, if $\sigma$ is any non-terminal extension of $\alpha_{n+1}$ in $T^L_1$ such that $y^*$ and $z^*$ do not share a column, then we delete from $T^L_1$ $\sigma$ and all of its extensions. As $\alpha_{n+1}$ is not a transition node, we now return to Case B.1.

We complete stage $s$ as follows. If added some set $P$ to $Y$, or if we defined $I_{s+1}$ to be the range of an infinite path through $T_0$ or $T_1$, we are done. For if we added some set $P$ to $Y$, then there will be an $(f\oplus P)$-computable set $p$-homogeneous for $\Phi^f$, but by Lemma 3.3 there will be no $(f\oplus P)$-computable set homogeneous for $f$. And if we defined $I_{s+1}$ to be the range of an infinite path through $T_0$ or $T_1$---say through $T_0$---then from the definition of that tree $\Gamma^{H^\Phi_0}$ does does not define an infinite set, let alone one homogeneous for $f$. 

Otherwise, we succeeded either in defining $\alpha_{n+1}$ for each non-terminal $\alpha_n$ in the sequence of nodes through $T^L_0$ or else in defining $\alpha_{n+1}$ for each non-terminal $\alpha_n$ in the sequence of nodes through $T^L_1$; say we succeeded in defining the sequence of nodes in $T^L_0$. That tree was in this case well-founded, so for some $i$, $\alpha_n$ was terminal. Then from the definition of the tree, there are some $F_L,F_R\subseteq\ran(\alpha_n)$ such that 
$\Gamma^{H^\Phi_{0,s}\cup F_L\oplus F_R}(a)\downarrow = \Gamma^{H^\Phi_{0,s}\cup F_L\oplus F_R}(b)\downarrow = 1$ 
for some unequal $a,b\ge |p_s|$, say with use $u$. Let $p_{s+1} = q_n$, $H^\Phi_{0,s+1} = H^\Phi_{0,s}\cup F_L\oplus F_R$, and $I_{s+1} = \{x\in I_s:x > u\}$.
\end{proof}


\section{Second Main Theorem}

\begin{theorem}
{\it There exists a stable 2-coloring of pairs $f:[\omega]^2\to 2$ and a family $Y$ of infinite sets such that no $(f\oplus P)$-computable set is $p$-homogeneous for $f$ for any $P\in Y$, and every stable 2-coloring of pairs $f':[\omega]^2\to 2$ computable from $f$ has either an $(f\oplus P)$-computable increasing $p$-homogeneous set for some $P\in Y$, or if not then some increasing $p$-homogeneous set which does not compute a set $p$-homogeneous for $f$. }
\end{theorem}

Before proving Theorem 4.1, we observe that the Second Main Theorem is a direct consequence. The explanation is the same as that given in section 3 for the deduction of the First Main Theorem from Theorem 3.1. 


We recall that $\mathbb{C}$ is the following notion of forcing: a condition is an ordered triple $p=\langle \sigma^p,l^p,|p| \rangle$ where $|p|\in\omega$, $\sigma^p:\left[|p|\right]^2\to 2$, $l^p:|p|\to 2\times\omega$, and $l^p(x) = \langle i,z \rangle$ implies that if $\sigma^p(x,y)$ is defined and $y\ge z$ then $\sigma^p(x,y) = i$. 

\begin{lemma}
If $p_0\ge p_1\ge p_2\ge \cdots$ is a sequence of $\mathbb{C}$-conditions which is 3-generic relative to some set $P\subseteq\omega$, and if $f = \bigcup_s \sigma^{p_s}:[\omega]^2\to 2$, then $f\oplus P$ does not compute a $p$-homogeneous set for $f$. 
\end{lemma}

\begin{proof}[Proof of Lemma 4.2] This follows from Lemma 3.3, since if $f\oplus P$ were to compute a $p$-homogeneous set for $f$, then from that set together with $f$ one could compute a homogeneous set for $f$.
\end{proof}

\begin{proof}[Proof of Theorem 4.1]
We build
\begin{itemize} 
\item a sequence of $\mathbb{C}$-conditions $p_0\ge p_1\ge p_2\ge\cdots$ with $\lim_s |p_s| = \infty$; 
\item sequences of finite sets (initial segments of increasing $p$-homogeneous sets) $H^\Phi_{j,0}\subseteq H^\Phi_{j,1}\subseteq H^\Phi_{j,2}\subseteq\cdots$ for each Turing functional $\Phi$ and each $j<2$; 
\item a sequence of infinite sets (reservoirs) $I_0\supseteq I_1\supseteq I_2\supseteq\cdots$ with $H^\Phi_{j,s} < I_s$ for each $\Phi,j,s$; 
\item a sequence of finite families $Y_0\subseteq Y_1\subseteq Y_2\subseteq\cdots$ of infinite subsets of $\omega$; 
\end{itemize} 
and we define $f = \bigcup_s \sigma^{p_s}$, $H^\Phi_j = \bigcup_s H^\Phi_{j,s}$ for each $j<2$, and $Y = \bigcup_s Y_s$. The construction will ensure the following requirements, for all $i\in\omega$ and Turing functionals $\Phi,\Gamma,\Delta$: 

\hspace*{9mm}\begin{center}\begin{tabular}{lcl} 
$\mathcal{P}_i$ & : & the sequence $p_0\ge p_1\ge p_2\ge \cdots$ is 3-generic relative to each $P\in Y_i$; \\ 
& & \\
$\mathcal{Q}_{\Phi,i}$ & : & if $\Phi^f$ is a stable 2-coloring of pairs, it either has an $(f\oplus P)$-computable \\ 
& &  increasing $p$-homogeneous set for some $P\in Y$ or else both $H^\Phi_0$ and $H^\Phi_1$ \\ & & are infinite in both columns; \\ 
& & \\
$\mathcal{R}^\Phi_{\Gamma,\Delta}$ & : & if $\Phi^f$ is a stable 2-coloring of pairs, it either has an $(f\oplus P)$-computable \\ 
& & increasing $p$-homogeneous set for some $P\in Y$; or else if $\Gamma^{H^\Phi_0}$ defines an \\ 
& & infinite set then this set is not $p$-homogeneous for $f$; or else if $\Delta^{H^\Phi_1}$ \\ 
& & defines an infinite set then this set is not $p$-homogeneous for $f$. 
\end{tabular}\end{center}

\underline{\it Construction.} 

Devote infinitely many stages $s\in\omega$ to each requirement. Let $p_0$ be any condition with $|p_s| = 0$. For each $\Phi$ let $H^\Phi_{0,0} = H^\Phi_{1,0} = \emptyset$, and let $I_0 = \omega$ and $Y_0 = \emptyset$. At stage $s+1$ assume by way of induction that we have $p_s$, $H^\Phi_{j,s}$ for $j<2$ and all $\Phi$, $I_s$, and $Y_s$ and assume that if $H^\Phi_{j,s}$ is nonempty for some $j$ and $\Phi$, then $p_s$ forces that $\Phi^f$ is a stable coloring of pairs and that for $x<y$, $\Phi^f(x,y) = j$ whenever $2x,2y+1\in H^\Phi_{j,s}$ or when $y\in I_s$ and $2x\in H^\Phi_{j,s}$. At the end of a stage any of $p_{s+1}$, $H^\Phi_{j,s+1}$, $I_{s+1}$, or $Y_{s+1}$ not yet defined should be taken to be identical to $p_s$, $H^\Phi_{j,s}$, $I_{s}$, or $Y_s$, respectively. 


\subsection*{$\mathcal{P}$ requirements.} Suppose $s$ is dedicated to requirement $\mathcal{P}_i$ for some $i < s$ and that it is the $\langle n,m \rangle^{\rm th}$ such stage. If $n > |Y_i|$ do nothing. Otherwise, let $P$ be the $n^{\rm th}$ member of the family $Y_i$ in some fixed enumeration and let $W$ be the $m^{\rm th}$ $\Sigma^0_3(P)$ set in some fixed enumeration. If $p_s$ has an extension $q$ in $W$, fix $q$ and let $p_{s+1} = q$, so that $p_0\ge p_1\ge \cdots \ge p_s\ge p_{s+1}\ge \cdots$ meets $W$. Otherwise, do nothing, and $p_0\ge p_1\ge \cdots \ge p_s\ge p_{s+1}\ge \cdots$ avoids $W$. 


\subsection*{$\mathcal{Q}$ requirements.} Suppose $s$ is dedicated to $\mathcal{Q}_{\Phi,i}$. Without loss of generality we assume that $p_s$ decides whether or not $\Phi^f$ is a stable 2-coloring of pairs. If $p_s$ forces that $\Phi^f$ is not such a coloring, do nothing. Otherwise, we consider two cases. 
\begin{itemize} 
\item If for some $j < 2$ and $k\in\omega$ there is no extension of $p_s$ which forces that $\lim_u\Phi^f(x,u) = j$ for some $x\ge k$ in $I_s$, then $P = \{x\in I_s:x\ge k\}$ is limit homogeneous for $\Phi^f$ with color $1-j$, and so $(f\oplus P)$ computes an increasing $p$-homogeneous set for $\Phi^f$. We set $Y_{s+1} = Y_s\cup\{P\}$. This satisfies the requirement. 
\item If no such $j,k$ exist, then there are numbers $x_{00},x_{01},x_{10},x_{11}\in I_s$ with $x_{00}<x_{01}$ and $x_{10} < x_{11}$ and an extension of $p_s$ forcing that $H^\Phi_{j,s}\cup\{2x_{j0},2x_{j1}+1\}$ is finite increasing $p$-homogeneous for $\Phi^f$ and $\lim_u\Phi^f(x_{ji},u) = j$ for each $i,j<2$. In this case let $p_{s+1}$ be such an extension of $p_s$, let $H^\Phi_{j,s+1} = H^\Phi_{j,s}\cup\{2x_{j0},2x_{j1}+1\}$, and let $I_{s+1} = \{x:\,m<x\in I_s\}$ where $m$ is greater than the stabilization points under $\Phi^f$ of every element of $H^\Phi_{0,s+1}\cup H^\Phi_{1,s+1}$. Observe that both columns have been extended by one element. 
\end{itemize}


\subsection*{$\mathcal{R}$ requirements.} Suppose $s$ is dedicated to $\mathcal{R}^{\Phi}_{\Gamma,\Delta}$ and assume that $p_s$ forces that $\Phi^f$ is a stable coloring of pairs. As before, the goal of this requirement is to extend by some finite set either the initial segment of $H^\Phi_0$ or the initial segment of $H^\Phi_1$ that we have constructed so far, subject to the following condition: if we extend $H^\Phi_0$, then $\Gamma$ does not compute a $p$-homogeneous set for $f$ from any further extension of $H^\Phi_0$; and if on the other hand we extend $H^\Phi_1$, then $\Delta$ does not compute a $p$-homogeneous set for $f$ from any further extension of $H^\Phi_1$. Just as in the proof of Theorem 3.1, the finite set by which we extend either $H^\Phi_0$ or $H^\Phi_1$ will be obtained as a subset of the range of a path through an infinitely branching tree. When we have so extended one of $H^\Phi_0,H^\Phi_1$, we will say, as before, that we have successfully diagonalized against such computations. 

The means by which we accomplish the diagonalization here will be slightly more involved. There will again be two numbers, say $a$ and $b$, and two finite sets $F_L$ and $F_R$, all arising from a tree labeling construction, about which we know either that 
$$\Gamma^{H^\Phi_0\cup (F_L\oplus F_R)}(2a+1)\downarrow = \Gamma^{H^\Phi_0\cup (F_L\oplus F_R)}(2b)\downarrow = 1$$ 
or else that 
$$\Delta^{H^\Phi_1\cup (F_L\oplus F_R)}(2a+1)\downarrow = \Delta^{H^\Phi_1\cup (F_L\oplus F_R)}(2b)\downarrow = 1.$$ 
Let us suppose we know the first. In other words, we know that if we extend $H^\Phi_0$ by $F_L\oplus F_R$, then the set computed from $H^\Phi_0\cup(F_L\oplus F_R)$ by $\Gamma$ will contain $2a+1$ and $2b$. When we view the set so computed as having two columns (i.e. as being the join of two sets), this means that $a$ appears in the right-hand column and $b$ appears in the left-hand column. Thus to ensure that the set so computed is not $p$-homogeneous for $f$, we will choose as before an extension $q$ of the condition $p_s$ such that $\sigma^q(a,b)$, $(l^q(a))_0$, and $(l^q(b))_0$ are not all equal. This will guarantee that no set containing both $a$ and $b$ is $p$-homogeneous for the coloring eventually obtained by extending $\sigma^q$.

As in the proof of Theorem 3.1, the construction divides into Case B and Case A, whose description is unchanged. That is, in Case B we suppose there is an $i<2$, a condition $q\le p_s$, an infinite set $I\subseteq I_s$, and a set $Q = \{\langle x,a_x,b_x \rangle: x\in I,\,a_x,b_x\in\omega\}$ such that there is exactly one triple in $Q$ with first element $x$ for each $x\in I$, $a_x < b_x$, and such that infinitely many numbers $b_x\in\omega$ appear as third elements of triples in $Q$; and we suppose these have the following property: for any $x\in I$ and any extension $r$ of $q$ such that $\sigma^r(a_x,b_x)$, $(l^r(a_x))_0$, and $(l^r(b_x))_0$ are not all equal will force that $\lim_u \Phi^f(x,u) = i$. In Case A we suppose this is not so.

\subsection*{Case A}

For this proof we need again to modify slightly the definition of tree labeling with two labels. We now say that a nonempty string $\alpha\in T(k,\Gamma,H,I)$ if $\alpha\in I^{<\omega}$ is increasing and there are no finite $F_L,F_R\subseteq\ran(\alpha^{\#})$ and no $b > a \ge k$ such that 
$$\Gamma^{H\cup (F_L\oplus F_R)}(2a+1)\downarrow = \Gamma^{H\cup (F_L\oplus F_R)}(2b)\downarrow = 1.$$
With this modification of the tree labeling definition in hand, we now as before let $T_0 = T(|p_s|,\Gamma,H^\Phi_{0,s},I_s)$. If $T_0$ is not well-founded then let $I_{s+1}$ be the range of an infinite path through $T_0$. Observe that in this case the requirement is satisfied. If $T_0$ is well-founded, then let $T^L_0$ be the labeled subtree of $T_0$. 

We now try to define two sequences, conditions 
$$p_s\ge q_0 \ge q_1 \ge q_2 \ge \cdots$$ 
and nodes of $T^L_0$ 
$$\emptyset = \alpha_0 \preceq \alpha_1 \preceq \alpha_2 \preceq \cdots$$ 
where for all $j\ge 0$, $\alpha_{j+1}$ is a successor of $\alpha_j$ and for all $j\ge 0$ the condition $q_j$ forces that 
$$\lim_u \Phi^f(x,u) = 0$$ 
for all $x\in\ran(\alpha_j)$. From here the proof is exactly the same as that of Case A in Theorem 3.1.

\subsection*{Case B}

Fix an $i<2$, a condition $q\le p_s$, an infinite set $I\subseteq I_s$, and a set $Q = \{\langle x,a_x,b_x \rangle: x\in I,\,a_x,b_x\in\omega\}$ such that there is exactly one triple in $Q$ with first element $x$ for each $x\in I$; such that infinitely many numbers $b_x\in\omega$ appear as third elements of triples in $Q$; and such that for any $x\in I$ and any extension $r$ of $q$ such that $\sigma^r(a_x,b_x)$, $(l^r(a_x))_0$, and $(l^r(b_x))_0$ are not all equal will force that $\lim_u \Phi^f(x,u) = i$. Suppose without loss of generality that $i=1$. Thus for any $x\in I$, any forcing extension which presses the button of $x$ forces that $\lim_u \Phi^f(x,u)=1$. 

Here again we must modify the tree labeling method, varying tree labeling with three labels. The definition of $T(k,\Gamma,H,I)$ is for this case changed to the following. $\emptyset \in T(k,\Gamma,H,I)$ and for a nonempty string $\alpha$, $\alpha\in T(k,\Gamma,H,I)$ if $\alpha\in I^{<\omega}$ is increasing and there are no finite $F_L,F_R\subseteq\ran(\alpha^{\#})$ and no $c > b > a \ge k$ such that 
$$\Gamma^{H\cup(F_L\oplus F_R)}(2a+1)\downarrow = \Gamma^{H\cup(F_L\oplus F_R)}(2b)\downarrow = \Gamma^{H\cup(F_L\oplus F_R)}(2c+1)\downarrow = 1.$$
The method for labeling the nodes of $T(k,\Gamma,H,I)$ extends the method from tree labeling with two labels in the natural way, as does the method for selecting the nodes of the labeled subtree $T^L(k,\Gamma,H,I)$. Remark 2.3 applies \textit{mutatis mutandis}. Finally, we use the revised the definition of a transition node from Case B in the proof of the previous theorem: a node $\alpha\in T^L$ is called a transition node if the symbol $\infty$ appears in the label of $\alpha$, and appears strictly fewer times \textit{and no more than once} in the label of each successor of $\alpha$. 

Now let $T_1 = T(|q|,\Delta,H^{\Phi}_{1,s},I_s)$. If $T_1$ is not well founded then let $I_{s+1}$ be the range of an infinite path through $T_1$. Observe that in this case the requirement is satisfied. If $T_1$ is well founded, let $T^L_1$ be the labeled subtree of $T_1$.

We now try as in the proof of Theorem 3.1 to define two sequences, conditions 
$$q\ge q_0\ge q_1\ge q_2\ge \cdots$$ 
and nodes of $T^L_1$ 
$$\emptyset=\alpha_0\preceq \alpha_1\preceq \alpha_2\preceq\cdots$$
where for all $j\ge 0$, $\alpha_{j+1}$ is successor of $\alpha_j$ and for all $j\ge 0$ the condition $q_j$ forces that 
$$\lim_u\Phi^f(x,u) = 0$$ 
for all $x\in\ran(\alpha_j)$. From here the proof is exactly the same as that of Theorem 3.1 up until Case B.2.3, where the diagonalization strategy changes to reflect the different combinatorics of the present theorem. We therefore refer the reader to the previous proof, and restart this proof at Case B.2.3 below. 

Recalling the function $\mathsf{Sort}$ defined in Case B.2.3 of the previous proof, we make a new definition. 
If $\alpha\in T^L_1$ is not terminal and $x,y\in\ran(\alpha)$, we define the following properties: 
\begin{itemize} 
	\item \textit{Configuration I} holds if there is $F_L\oplus F_R\in\mathsf{Sort}(\alpha)$ such that $x\in F_R$ and $y\in F_L$. 
	\item \textit{Configuration II} holds if there is $F_L\oplus F_R\in\mathsf{Sort}(\alpha)$ such that $x,y\in F_L$ or $x,y\in F_R$ or one of $x,y$ appears in neither of $F_L,F_R$. 
	\item \textit{Configuration III} holds if there is $F_L\oplus F_R\in\mathsf{Sort}(\alpha)$ such that $x\in F_L$ and $y\in F_R$. 
\end{itemize} 

We now attend to Case B.2.3. 

\subsection*{Case B.2.3.} $\langle a,\infty,\infty \rangle \to \langle a,b,\infty \rangle$ \\ 

If $\alpha_n$ has label $\langle a^*, \infty, \infty \rangle$ and every successor of $\alpha_n$ has a label in which the symbol $\infty$ appears exactly once, then we let 
$$P = \{y:\alpha_n*y\in S\land \forall j<n\,(b > b_{\alpha_n(j)})\land b > |q_n|)\},$$ 
where $b$ here denotes the second entry in the label of $\alpha_n*y$ and $b_{\alpha_n(j)}$ denotes the button of $\alpha_n(j)$. Whether there is much work to be done in this case depends on whether the first node to have label $\langle a^*,\infty,\infty \rangle$ was or was not the root node of $T^L_1$. Formally, suppose $k$ is the least index such that $\alpha_k$ has a label in which the symbol $\infty$ appears exactly twice. 
\begin{itemize}
\item If, on the one hand, $\alpha_k = \alpha_0 = \emptyset$, then we choose any $y^*\in P$ and any extension $r$ of $q_n$ which presses $b_{y^*}$ and is such that, if the label of $\alpha_n*y^*$ is $\langle a^*,b^*,\infty \rangle$, then $\sigma^r(a^*,b^*)$, $(l^r(a^*))_0$, and $(l^r(b^*))_0$ are not all equal. In this case let $q_{n+1} = r$ and $\alpha_{n+1} = \gamma$ for any $\gamma\in S$ having $\gamma(n) = y^*$ with label $\langle a^*,b^*,\infty\rangle$. 

\item If, on the other hand, $k > 0$ and $\alpha_k\ne\emptyset$, then there is more work to do. Suppose that $\alpha_k(k-1) = x^*$. Let $P'\subseteq P$ contain precisely the elements $y$ of $P$ such that either Configuration I or Configuration II holds between $x^*$ and $y$. Either $P'\ne\emptyset$ and we complete the diagonalization in this case, or else $P'=\emptyset$ and we wait until the next case to diagonalize, but we are then guaranteed to succeed.
\end{itemize}
The important idea here is intuitively as follows. Either we may choose $y^*$ from $P$ so that Configuration I or Configuration II obtains, or else $x^*$ inhabits the left-hand column and $y^*$ inhabits the right-hand column and then $z^*$ may be chosen at the next transition node so that $\Phi^f(y^*,z^*)$ can be safely changed without disturbing the construction. 

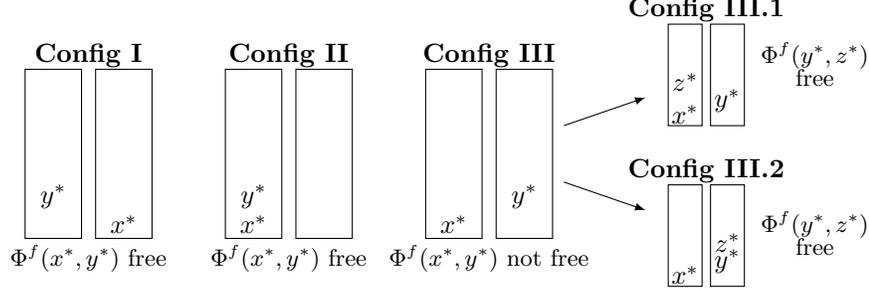
\hspace*{-5mm}\begin{tikzpicture}[scale=0.75]
\node at (1.125,3.25) {\bf Config I};
\node at (1.125,-0.35) {\small $\Phi^f(x^*,y^*)$ free};
\draw (0,0) rectangle (1,3); \node at (0.5,0.75) {$y^*$};
\draw (1.25,0) rectangle (2.25,3); \node at (1.75,0.25) {$x^*$};
\begin{scope}[xshift=-0.45cm,yshift=0cm]
\node at (5.125,3.25) {\bf Config II};
\node at (5.125,-0.35) {\small $\Phi^f(x^*,y^*)$ free};
\draw (4,0) rectangle (5,3); \node at (4.5,0.75) {$y^*$};
\draw (5.25,0) rectangle (6.25,3); \node at (4.5,0.25) {$x^*$};
\end{scope}
\begin{scope}[xshift=-0.9cm,yshift=0cm]
\node at (9.125,3.25) {\bf Config III};
\node at (9.125,-0.35) {\small $\Phi^f(x^*,y^*)$ not free};
\draw (8,0) rectangle (9,3); \node at (8.5,0.25) {$x^*$};
\draw (9.25,0) rectangle (10.25,3); \node at (9.75,0.75) {$y^*$};
\end{scope}
\draw[-latex] (9.55,1) -- (11,0.5);
\draw[-latex] (9.55,2) -- (11,2.5);
	%
	\begin{scope}[xshift=11.4cm,yshift=2cm,scale=0.6]
	\node at (1.125,3.42) {\bf Config III.1};
	\draw (0,0) rectangle (1,3); \node at (0.5,0.34) {$x^*$};
	\draw (1.25,0) rectangle (2.25,3); \node at (1.75,0.75) {$y^*$};
	\node at (0.5,1.29) {$z^*$};
	\end{scope}
	%
	\begin{scope}[xshift=11.4cm,yshift=-0.85cm,scale=0.6]
	\node at (1.125,3.42) {\bf Config III.2};
	\draw (0,0) rectangle (1,3); \node at (0.5,0.34) {$x^*$};
	\draw (1.25,0) rectangle (2.25,3); \node at (1.75,0.75) {$y^*$};
	\node at (1.75,1.29) {$z^*$};
	\end{scope}
\node at (14.0,0.25) {\small $\Phi^f(y^*,z^*)$};
\node at (14.0,-0.15) {\small free};
\node at (14.0,3.25) {\small $\Phi^f(y^*,z^*)$};
\node at (14.0,2.85) {\small free};
\end{tikzpicture}


If $P'$ is nonempty, then we choose any $y^*\in P'$ with label $\langle a^*,b^*,\infty\rangle$ and a condition $r$ which extends $q_n$ except possibly having $(l^r(a^*))_1\ne (l^{q_n}(a^*))_1$ if the latter is defined; and we choose $y^*$ and $r$ such that $r$ presses $b_{y^*}$ and $\sigma^r(a^*,b^*)$, $(l^r(a^*))_0$, and $(l^r(b^*))_0$ are not all equal; and we let $q_{n+1} = r$ and $\alpha_{n+1} = \gamma$ for any $\gamma\in S$ having $\gamma(n) = y^*$ with label $\langle a^*,b^*,\infty\rangle$. Finally, if $\sigma$ is any non-terminal extension of $\alpha_{n+1}$ in $T^L_1$ such that Configuration III holds between $x^*$ and $y^*$, then we delete from $T^L_1$ $\sigma$ and all of its extensions. As $\alpha_{n+1}$ is not a transition node, we now return to Case B.1.

\setlength{\leftskip}{0pt}
\setlength{\rightskip}{0pt}

If on the other hand $P'$ is empty, then we proceed as in the non-transition case and will diagonalize in Case B.2.4 instead. 
\subsection*{Case B.2.4.} $\langle a,b,\infty \rangle \to \langle a,b,c \rangle$ \\ 

If $\alpha_n$ has label $\langle a^*,b^*,\infty\rangle$ and every successor of $\alpha_n$ has a label in which only finite numbers occur, and if we failed to diagonalize at an earlier node in Case B.2.3, then we proceed as follows; otherwise we proceed as in the non-transition case. Let 
$$P = \{z:\alpha_n*z\in S\land \forall j<n\,(c > b_{\alpha_n(j)})\land c > |q_n|)\},$$ 
where $c$ here denotes the third entry in the label of $\alpha_n*z$ and $b_{\alpha_n(j)}$ denotes the button of $\alpha_n(j)$. Suppose $k$ is the least index such that $\alpha_k$ has a label in which the symbol $\infty$ appears exactly twice and that $l$ is the least index such that $\alpha_l$ has a label in which the symbol $\infty$ appears exactly once. Suppose that $\alpha_k(k-1) = x^*$ and that $\alpha_l(l-1) = y^*$. Let $P'\subseteq P$ contain precisely those elements of $P$ such that $x^*,z$ are in Configuration I and $P''\subseteq P$ contain precisely those elements of $P$ such that $y^*,z$ are in Configuration II. At least one of $P',P''$ must be nonempty; without loss of generality we assume that $P''$ is nonempty. Then we choose $z^*\in P''$ with label $\langle a^*,b^*,c^*\rangle$ and a condition $r$ which extends $q_n$ except possibly having $(l^r(b^*))_1\ne (l^{q_n}(b^*))_1$ if the latter is defined; and we choose $z^*$ and $r$ such that $\sigma^r(b^*,c^*)$, $(l^r(b^*))_0$, and $(l^r(c^*))_0$ are not all equal and $r$ presses $b_{z^*}$. We let $q_{n+1} = r$ and $\alpha_{n+1} = \gamma$ for any $\gamma\in S$ having $\gamma(n) = z^*$ with label $\langle a^*,b^*,c^*\rangle$. Such $z^*$ and $r$ exist by the same reasoning given in the previous case. Finally, if $\sigma$ is any non-terminal extension of $\alpha_{n+1}$ in $T^L_1$ such that Configuration II does not hold between $y^*$ and $z^*$, then we delete from $T^L_1$ $\sigma$ and all of its extensions. As $\alpha_{n+1}$ is not a transition node, we now return to Case B.1.

We complete stage $s$ as follows. If added some set $P$ to $Y$, or if we defined $I_{s+1}$ to be the range of an infinite path through $T_0$ or $T_1$, we are done. Otherwise, we succeeded either in defining $\alpha_{n+1}$ for each non-terminal $\alpha_n$ in the sequence of nodes through $T^L_0$ or else in defining $\alpha_{n+1}$ for each non-terminal $\alpha_n$ in the sequence of nodes through $T^L_1$; say we succeeded in defining the sequence of nodes in $T^L_0$. This tree was in this case well-founded, so for some $i$, $\alpha_n$ was terminal. Then from the definition of the tree, there are some $F_L,F_R\subseteq\ran(\alpha_n)$ such that 
$\Gamma^{H^\Phi_{0,s}\cup (F_L\oplus F_R)}(2a+1)\downarrow = \Gamma^{H^\Phi_{0,s}\cup (F_L\oplus F_R)}(2b)\downarrow = \Gamma^{H^\Phi_{0,s}\cup (F_L\oplus F_R)}(2c+1)\downarrow = 1$ 
for some unequal $a,b,c\ge |p_s|$, say with use $u$. Let $p_{s+1} = q_i$, $H^\Phi_{0,s+1} = H^\Phi_{0,s}\cup F_L\oplus F_R$, and $I_{s+1} = \{x\in I_s:x > u\}$.
\end{proof}

\section{Summary}

The diagram below records all the reductions between the principles studied in this paper. In the diagram, we write $Q\rightarrow P$ to mean that problem $P$ reduces to problem $Q$ in the indicated sense. 


\begin{center}
\begin{tikzpicture} 


\pgfmathsetmacro{\vertical}{1.75}
\pgfmathsetmacro{\horizontal}{5.25}
\pgfmathsetmacro{\xshift}{0.12}


\node (c) at ({0*\horizontal},{3.5*\vertical}) {\Large $\boxed{\le_{\rm c}}$};
\node (w) at ({1*\horizontal},{3.5*\vertical}) {\Large $\boxed{\le_{\rm W}}$};
\node (strong) at ({2*\horizontal},{3.5*\vertical}) {\Large $\boxed{\le_{\rm sc}\, ,\,\le_{\rm sW}}$};

\node (SRTc) at ({0*\horizontal},{3*\vertical}) {$\mathsf{SRT}^2_2$}; \node (SRTsc) at ({2*\horizontal},{3*\vertical}) {$\mathsf{SRT}^2_2$};
\node (SRTw) at ({1*\horizontal},{3*\vertical}) {$\mathsf{SRT}^2_2$}; 
\node (SRTcl) at ({0*\horizontal-\xshift},{3*\vertical}) {${}^{~}_{~}$}; \node (SRTscl) at ({2*\horizontal-\xshift},{3*\vertical}) {${}^{~}_{~}$};
\node (SRTwl) at ({1*\horizontal-\xshift},{3*\vertical}) {${}^{~}_{~}$}; 
\node (SRTcr) at ({0*\horizontal+\xshift},{3*\vertical}) {${}^{~}_{~}$}; \node (SRTscr) at ({2*\horizontal+\xshift},{3*\vertical}) {${}^{~}_{~}$};
\node (SRTwr) at ({1*\horizontal+\xshift},{3*\vertical}) {${}^{~}_{~}$}; 

\node (SPTc) at ({0*\horizontal},{2*\vertical}) {$\mathsf{SPT}^2_2$}; \node (SPTsc) at ({2*\horizontal},{2*\vertical}) {$\mathsf{SPT}^2_2$};
\node (SPTw) at ({1*\horizontal},{2*\vertical}) {$\mathsf{SPT}^2_2$}; 
\node (SPTcl) at ({0*\horizontal-\xshift},{2*\vertical}) {${}^{~}_{~}$}; \node (SPTscl) at ({2*\horizontal-\xshift},{2*\vertical}) {${}^{~}_{~}$};
\node (SPTwl) at ({1*\horizontal-\xshift},{2*\vertical}) {${}^{~}_{~}$}; 
\node (SPTcr) at ({0*\horizontal+\xshift},{2*\vertical}) {${}^{~}_{~}$}; \node (SPTscr) at ({2*\horizontal+\xshift},{2*\vertical}) {${}^{~}_{~}$};
\node (SPTwr) at ({1*\horizontal+\xshift},{2*\vertical}) {${}^{~}_{~}$}; 

\node (SIPTc) at ({0*\horizontal},{1*\vertical}) {$\mathsf{SIPT}^2_2$}; \node (SIPTsc) at ({2*\horizontal},{1*\vertical}) {$\mathsf{SIPT}^2_2$};
\node (SIPTw) at ({1*\horizontal},{1*\vertical}) {$\mathsf{SIPT}^2_2$}; 
\node (SIPTcl) at ({0*\horizontal-\xshift},{1*\vertical}) {${}^{~}_{~}$}; \node (SIPTscl) at ({2*\horizontal-\xshift},{1*\vertical}) {${}^{~}_{~}$};
\node (SIPTwl) at ({1*\horizontal-\xshift},{1*\vertical}) {${}^{~}_{~}$}; 
\node (SIPTcr) at ({0*\horizontal+\xshift},{1*\vertical}) {${}^{~}_{~}$}; \node (SIPTscr) at ({2*\horizontal+\xshift},{1*\vertical}) {${}^{~}_{~}$};
\node (SIPTwr) at ({1*\horizontal+\xshift},{1*\vertical}) {${}^{~}_{~}$}; 

\node (Dc) at ({0*\horizontal},{0*\vertical}) {$\mathsf{D}^2_2$}; \node (Dsc) at ({2*\horizontal},{0*\vertical}) {$\mathsf{D}^2_2$};
\node (Dw) at ({1*\horizontal},{0*\vertical}) {$\mathsf{D}^2_2$}; 
\node (Dcl) at ({0*\horizontal-\xshift},{0*\vertical}) {${}^{~}_{~}$}; \node (Dscl) at ({2*\horizontal-\xshift},{0*\vertical}) {${}^{~}_{~}$};
\node (Dwl) at ({1*\horizontal-\xshift},{0*\vertical}) {${}^{~}_{~}$}; 
\node (Dcr) at ({0*\horizontal+\xshift},{0*\vertical}) {${}^{~}_{~}$}; \node (Dscr) at ({2*\horizontal+\xshift},{0*\vertical}) {${}^{~}_{~}$};
\node (Dwr) at ({1*\horizontal+\xshift},{0*\vertical}) {${}^{~}_{~}$}; 


\draw[-latex] (SRTcl) -- (SPTcl); 
\draw[-latex] (SPTcl) -- (SIPTcl); 
\draw[-latex] (SIPTcl) -- (Dcl);
\draw[-latex] (Dcr) -- (SIPTcr); 
\draw[-latex] (SIPTcr) -- (SPTcr); 
\draw[-latex] (SPTcr) -- (SRTcr);


\draw[-latex] (SRTsc) -- (SPTsc); 
\draw[-latex] (SPTsc) -- (SIPTsc); 
\draw[-latex] (SIPTsc) -- (Dsc);


\draw[-latex] (SIPTw) -- (Dw);
\draw[-latex] (SRTwl) -- (SPTwl); 
\draw[-latex] (SPTwl) -- (SIPTwl);
\draw[-latex] (SPTwr) -- (SRTwr);
\draw[-latex] (SIPTwr) -- (SPTwr);



\end{tikzpicture}
\end{center}

\bibliography{proofs}{}
\bibliographystyle{amsplain}

\end{document}